\documentclass[11pt,a4paper]{amsart}
\usepackage{amssymb,xspace}
\usepackage{amstext}
\theoremstyle{plain}
\usepackage{amsbsy,amssymb,amsfonts,latexsym}

\usepackage{enumerate}
\usepackage{graphics}

\date{\today}
\title{Multifractal analysis of the divergence of Fourier series}
\author{Fr\'ed\'eric Bayart, Yanick Heurteaux}
\address{
Clermont Universit\'e, Universit\'e Blaise Pascal, Laboratoire de Math\'ematiques, BP 10448, F-63000 CLERMONT-FERRAND -
CNRS, UMR 6620, Laboratoire de Math\'ematiques, F-63177 AUBIERE
}
\email{Frederic.Bayart@math.univ-bpclermont.fr, Yanick.Heurteaux@math.univ-bpclermont.fr}



\newcommand{\veps}{\varepsilon}

\def\NN{\mathbb N}

\def\TT{\mathbb T}
\def\DD{\mathbb D}

\def\H{\mathcal H}

\newtheorem{theorem}{Theorem}[section]

\newtheorem{lemma}[theorem]{Lemma}

\newtheorem{corollary}[theorem]{Corollary}

{\theoremstyle{definition}}
{\theoremstyle{definition}}

{\theoremstyle{definition}}

{\theoremstyle{definition}\newtheorem{definition}[theorem]{Definition}}

{\theoremstyle{definition}}

{\theoremstyle{definition}\newtheorem{remark}[theorem]{Remark}}

\begin{document}

\begin{abstract}
A famous theorem of Carleson says that, given any function $f\in L^p(\TT)$, $p\in(1,+\infty)$, its 
Fourier series $(S_nf(x))$ converges for almost every $x\in \mathbb T$. Beside this property, the 
series may diverge at some point, without exceeding  $O(n^{1/p})$. 
We define the divergence index at $x$ as the infimum of the positive real numbers $\beta$ such that
$S_nf(x)=O(n^\beta)$ and we are interested in the size of the exceptional sets $E_\beta$, namely the sets of $x\in\mathbb T$ with divergence index equal to $\beta$. We show that quasi-all functions in $L^p(\TT)$ have a multifractal behavior with respect to this definition.  Precisely, for quasi-all functions in $L^p(\mathbb T)$, for all $\beta\in[0,1/p]$, $E_\beta$ has
 Hausdorff dimension equal to $1-\beta p$.
We also investigate the same problem in $\mathcal C(\mathbb T)$, replacing polynomial divergence by
logarithmic divergence. In this context, the results that we get on the size of the exceptional sets 
are rather surprizing.\\
\end{abstract}

\maketitle

\section{Introduction}
\subsection{Description of the results}
The famous theorem of Carleson and Hunt asserts that, when $f$ belongs to $L^p(\mathbb T)$, $1<p<+\infty$,
where $\mathbb T=\mathbb R/\mathbb Z$, the sequence of the partial sums of its Fourier series $(S_nf(x))_{n\geq 0}$ 
converges for almost every $x\in\mathbb T$. On the other hand, it can diverge at some point. This divergence cannot be too 
fast since, for any $f\in L^p(\mathbb T)$ and any $x\in\mathbb T$, $|S_n
f(x)|\leq C_p n^{1/p}\Vert f\Vert_p$ (see \cite{Zyg} for instance). In view of these results, a natural question arises. How big can be the sets $F$ such that $|S_nf(x)|$ grows as fast
as possible for every $x\in F$? More generally, can we say something on the size of the sets such that $|S_nf(x)|$
behaves like (or as bad as) $n^{\beta}$ for some $\beta\in(0,1/p]$?

\smallskip

To measure the size of subsets of $\TT$, 
we shall use the Hausdorff dimension. Let us recall the relevant definitions (we refer to \cite{Falc} and to \cite{Mat95} for more on this subject). 
If $\phi:\mathbb R_+\to\mathbb R_+$ is a nondecreasing continuous function
satisfying $\phi(0)=0$ ($\phi$ is called a \emph{dimension function} or a \emph{jauge function}),
the $\phi$-Hausdorff outer measure of a set $E\subset \mathbb R^d$ is 
$$\mathcal H^{\phi}(E)=\lim_{\veps\to 0}\inf_{r\in R_\veps(E)}\sum_{B\in r}\phi(|B|),$$
$R_\veps(E)$ being the set of countable coverings of $E$ with balls $B$ of diameter $|B|\leq\veps$. 
When $\phi_s(x)=x^s$, we write for short $\mathcal H^s$ instead of $\mathcal H^{\phi_s}$. The Hausdorff dimension of a set $E$
is 
$$\dim_{\mathcal H}(E):=\sup\{s>0; \mathcal H^s (E)>0\}=\inf\{s>0;\ \mathcal H^s(E)=0\}.$$

The first result studying the Hausdorff dimension of the divergence sets of
 Fourier series is due to J-M. Aubry \cite{Aub06}.

\begin{theorem}\label{THMAUBRY}
Let $f\in L^p(\mathbb T)$, $1<p<+\infty$. For $\beta\geq 0$, define
$$\mathcal E(\beta,f)=\left\{x\in\mathbb T;\ \limsup_{n\to+\infty}n^{-\beta}|S_nf(x)|>0\right\}.$$
Then $\dim_\mathcal{H}\big(\mathcal E(\beta,f)\big)\leq 1-\beta p$. Conversely, given a set $E$ such that $\dim_\mathcal{H}(E)<1-\beta p$, there exists
a function $f\in L^p(\mathbb T)$ such that, for any $x\in E$,
$\displaystyle\limsup_{n\to +\infty}n^{-\beta}|S_nf(x)|=+\infty$.
\end{theorem}

This result motivated us to introduce the notion of divergence index. For a
given function $f\in L^p(\TT)$ and a given point $x_0\in \TT$, we can define
the real number $\beta(x_0)$ as the infimum of the non negative real numbers
$\beta$ such that $|S_nf(x_0)|=O(n^\beta)$. The real number $\beta(x_0)$ will
be called the \emph{divergence index} of the Fourier series of $f$ 
at point $x_0$. Of
course, for any function $f\in L^p(\TT)$ ($1<p<+\infty$) and any point $x_0\in\TT$, $0\le
\beta(x_0)\le 1/p$. Moreover, Carleson's theorem implies that $\beta(x_0)=0$
almost surely and we would like to have precise estimates on the size of the level sets of the function
$\beta$. These are defined as

\begin{eqnarray*}
E(\beta,f)&=&\left\{x\in\mathbb T;\ \beta(x)=\beta\right\}\\
&=&\left\{x\in\mathbb T;\ \limsup_{n\to+\infty}\frac{\log |S_nf(x)|}{\log
    n}=\beta\right\}.
\end{eqnarray*}
We  can ask for which values of $\beta$ the sets
$E(\beta,f)$ are non-empty. This set of values will be called the domain of definition of the spectrum of singularities of $f$.
If $\beta$ belongs to the domain of definition of the spectrum of singularities, it is also interesting to estimate the
Hausdorff dimension of the sets $E(\beta,f)$. The function $\beta\mapsto
\dim_\mathcal{H}(E(\beta,f))$ will be called the spectrum of singularities of the
function $f$ (in terms of its Fourier series). 
By Aubry's result, $\dim_\mathcal{H}(E(\beta,f))\leq1-\beta p$ and, for any fixed $\beta_0\in[0,1/p)$,
for any $\veps>0$, one can find 
$f\in L^p(\TT)$ such that $\dim_\mathcal{H}\left(\bigcup_{\beta_0\leq\beta\leq
  1/p}E(\beta,f)\right)\geq 1-\beta_0p-\veps$. Our first main result
is that a \emph{typical} function $f\in L^p(\TT)$ satisfies $\dim_\mathcal{H}(E(\beta,f))=1-\beta p$ for \emph{any} $\beta\in[0,1/p]$.
In particular, $f$ has a multifractal behavior with respect to the summation of its Fourier series, meaning that the domain of 
definition of its spectrum of singularities contains an interval with non-empty interior.
\begin{theorem}\label{THMLP}
Let $1<p<+\infty$. For quasi-all functions $f\in L^p(\TT)$, for any $\beta\in[0,1/p]$, $\dim_\mathcal{H}\big(E(\beta,f)\big)=1-\beta p$.
\end{theorem}
The terminology "quasi-all" used here is relative to the Baire category theorem. It means that this property is true for a residual
set of functions in $L^p(\TT)$.

\medskip

In a second part of the paper, we turn to the case of $\mathcal C(\mathbb
T)$, the set of continuous functions on $\mathbb T$. 
In that space, the divergence of Fourier series is controlled by a 
logarithmic factor.
More precisely, if $(D_n)$ is the sequence of the Dirichlet kernels, we know that $\|S_n
f\|_\infty\leq \|D_n\|_1\|f\|_\infty$, so that
there exists some absolute constant $C>0$ such that
$\|S_n f\|_\infty\leq  C\|f\|_\infty\log n$ for any $f\in\mathcal C(\mathbb T)$ and any $n> 1$.
As before, one can discuss the size of the sets such that $|S_nf(x)|$ behaves badly, namely like $(\log n)^{\beta}$,
$\beta\in[0,1]$. More precisely, mimicking the case of the $L^p$ spaces, 
we introduce, for any $\beta\in[0,1]$ and any $f\in\mathcal C(\mathbb T)$, the following sets:
\begin{eqnarray*}
\mathcal F(\beta,f)&=&\left\{x\in\mathbb T;\ \limsup_{n\to+\infty}\,(\log n)^{-\beta}|S_nf(x)|>0\right\}\\
F(\beta,f)&=&\left\{x\in\mathbb T;\ \limsup_{n\to+\infty}\frac{\log |S_nf(x)|}{\log\log n}=\beta\right\}.
\end{eqnarray*}
Theorem \ref{THMAUBRY} indicates that, on $L^p(\TT)$, $|S_nf(x)|$ can grow as fast as possible (namely like $n^{1/p}$) 
only on small sets: for every function $f\in L^p(\TT)$,
$\dim_{\mathcal H}(E(1/p,f))=0$. This property dramatically breaks down on $\mathcal C(\TT)$, as the following result indicates.
\begin{theorem}\label{THMCT1}
For quasi-all functions $f\in\mathcal C(\TT)$, $\dim_{\mathcal H}\big(F(1,f)\big)=1$.
\end{theorem}
Thus, for quasi-all functions $f\in\mathcal C(\TT)$, the partial sums
$(S_nf(x))_{n\ge 0}$ grow as fast as possible on big sets.

We can also study the domain of the spectrum of singularities of $f$, namely the values of $\beta$ such that $F(\beta,f)$ is non-empty.
Like in the case of the space $L^p(\TT)$, this domain is for quasi-all functions of $\mathcal C(\mathbb T)$ an interval with non-empty
interior, so that a typical function $f$ in $\mathcal C(\mathbb T)$ has a multifractral behavior with respect to the summation
of its Fourier series. 
However, the spectrum of singularities is constant!
\begin{theorem}\label{THMCT2}
For quasi-all functions $f\in\mathcal C(\TT)$, for any $\beta\in[0,1]$, $F(\beta,f)$ is non-empty
and has Hausdorff dimension 1.
\end{theorem}

Theorem \ref{THMCT2} indicates that the Hausdorff dimension is not precise enough to measure the size of the level sets $F(\beta,f)$.
This leads us to introduce a notion of \emph{precised Hausdorff dimension}, in
order to distinguish more finely sets which have the same
Hausdorff dimension. For $s> 0$ and $t\in(0,1]$, we consider 
$$\phi_{s,t}(x)=x^s \exp\left[(\log 1/x)^{1-t}\right].$$

\begin{definition}
Let $E\subset\mathbb R^d$. We say that $E$ has \emph{precised Hausdorff
  dimension} $(\alpha,\beta)$ if $\alpha$ is the Hausdorff dimension of $E$ and  
\begin{itemize}
\item $\beta=0$ if $\mathcal H^{\phi_{\alpha,t}}(E)=0$ for every $t\in(0,1)$;
\item $\beta=\sup\big\{t\in(0,1);\ \mathcal H^{\phi_{\alpha,t}}(E)>0\big\}$ otherwise.
\end{itemize}
\end{definition}
It is not difficult to check that $\phi_{s,t}(x)\leq\phi_{s',t'}(x)$ for small values of $x$ iff
$$s>s'\textrm{ or }(s=s'\textrm{ and }t\geq t').$$
Thus the precised Hausdorff dimension is a refinement of the Hausdorff
dimension. In particular it is a tool to classify sets that have the same
Hausdorff dimension. The natural order for the precised dimension $(s,t)$ is
the lexicographical order which will be denoted by $\prec$. 
With respect to this order, we can say that the
greater is the set, the greater is the precised dimension. Moreover, if
$(s,t)\prec (s',t')$ and $(s,t)\not= (s',t')$, then  $\phi_{s',t'}\ll
\phi_{s,t}$. It follows that 
$H^{\phi_{s',t'}}(E)=0$ as soon as $H^{\phi_{s,t}}(E)<\infty$.

\smallskip

Our main theorem on $\mathcal C(\mathbb T)$, which contains both Theorems \ref{THMCT1} and \ref{THMCT2}, is the following:
\begin{theorem}\label{THMCT3}
For quasi-all functions $f\in\mathcal C(\mathbb T)$, for any $\beta\in[0,1]$, the precised Hausdorff dimension of $F(\beta,f)$ is $(1,1-\beta)$.
\end{theorem}
\medskip

The paper is organized as follows. In the remaining part of this section, we introduce tools which will be needed during the rest of the paper.
In Section \ref{SECLP}, we prove Theorem \ref{THMLP} whereas in Section \ref{SECCT}, we prove Theorem \ref{THMCT3}. 

\subsection{A precised version of Fej\'er's theorem}

Working on Fourier series, we will need results on approximation by trigonometric polynomials.
Let $k\in\mathbb Z$ and $e_k:t\mapsto e^{2\pi i kt}$, so that, for any $g\in L^1(\TT)$
and any $n\in\mathbb N$, 
$$S_ng:t\mapsto \sum_{k=-n}^n \langle g,e_k\rangle e_k(t).$$
Let $\sigma_ng$ be the $n-$th Fej\'er sum of $g$,
$$\sigma_n g:t\mapsto \frac 1n\sum_{k=0}^{n-1}S_kg(t).$$
$\sigma_n g$ is obtained by taking the convolution of $g$ with the Fej\'er kernel 
$$F_n:t\mapsto \frac 1n\left(\frac{\sin(n\pi t)}{\sin(\pi t)}\right)^2.$$
If $g$ belongs to $\mathcal C(\TT)$, $(\sigma_n g)_{n\ge 1}$ converges uniformly to $g$. For our purpose, we need to estimate how quick is the convergence. This is the content
of the next lemma (part (1) rectifies a mistake in the proof of Lemma 12
in \cite{Aub06} and requires to replace $\|\theta\|_\infty/4$ in Aubry's
version  by $\|\theta\|_\infty/2$).

\begin{lemma}\label{LEMFEJER}Let $\theta$ be a Lipschitz
    function on $\TT$, let $n\in\mathbb N$ and let $x\in\mathbb T$.  Suppose that $\|\theta'\|_\infty\le n$
    and that $\theta(x)=0$. Then the two following inequalities hold:
\begin{eqnarray}\label{EQFEJER1}
|\sigma_{n}\theta(x)|\le\frac14+\frac12\|\theta\|_\infty\ \textrm{ for any }n\geq 8\\
|\sigma_{n}\theta(x)|\le 4+\frac14\|\theta\|_\infty\ \textrm{ for any }n\geq 4.\label{EQFEJER2}
\end{eqnarray}
\end{lemma}
\begin{proof}
 We may assume that $x=0$. Hence,
$\sigma_{n}\theta(0)=\int_{-1/2}^{1/2}\theta(y)F_{n}(y)\,dy$. Let us consider $\delta\in(0,2]$ and $n\ge 4$.
On the one hand, for any $y\in[0,1/2)$,
$$0\leq F_{n}(y)=\frac{\sin^2(n\pi y)}{n\sin^2(\pi
  y)}\leq\frac1{n(2y)^2}$$ 
so that 
$$\left|\int_{\delta/n<|y|\le1/2}\theta(y)F_{n}(y)\,dy\right|\leq
\frac{1}{2n}\|\theta\|_\infty\int_{\delta/n}^{+\infty}\frac{dy}{y^2}=\frac{\|\theta\|_\infty}{2\delta}.$$
On the other hand, 
$$\left|\int_{-\delta/n}^{\delta/n}\theta(y)F_{n}(y)\,dy\right|\leq
2\int_0^{\delta/n}\left(\frac{\sin (n\pi y)}{\sin(\pi y)}\right)^2y\,dy := u_n.$$
Using the convexity inequality $\sin\left(\frac{n}{n+1}\pi y\right)\geq
\frac{n}{n+1}\sin(\pi y)$ and a change of variables, we see that $(u_n)$ is
non-increasing. To prove (\ref{EQFEJER1}), we choose $\delta=1$ and we observe that $u_8=0.2496...\leq \frac14$.
To prove (\ref{EQFEJER2}), we choose $\delta=2$ and we observe that, since the maximum of $F_n$ is $F_n(0)=n$,
$$|u_n|\leq 2n^2\int_0^{2/n} ydy=4.$$
\end{proof}

\subsection{The mass transference principle}

The second main tool that we need in this paper is a method to produce sets with large Hausdorff dimension (Theorem \ref{THMLP})
or with large precised Hausdorff dimension (Theorem \ref{THMCT3}). An efficient way to do this
is to consider \emph{ubiquitous systems} like this was done in \cite{DMPV95,Jaf00b}. This was later refined in \cite{BV06} to obtain a general mass transference principle, which we recall in the form that we need.

% \begin{theorem}\label{THMUBIQUITY}
% Let $(x_n)_{n\ge 0}$ be a sequence of points in $[0,1]^d$ and let $(r_n)_{n\ge 0}$
% be a sequence of positive real numbers decreasing to 0. For every 
% $\alpha\ge 1$, define
% $$E_\alpha=\limsup_n B(x_n,r_n^\alpha)=\bigcap_{N\ge 0}\bigcup_{n\ge
%   N}B(x_n,r_n^\alpha)$$ 
% and suppose that almost every point of $[0,1]^d$ (in the sense of the
% Lebesgue's measure) lies in $E_1$. Then,
% $$\forall\alpha\ge 1,\quad \dim_\mathcal{H}(E_\alpha)\ge\frac{d}\alpha.$$
% More precisely, if $\psi(s)=s^{d/\alpha}(\log 1/s)^2$, $\mathcal{H}^\psi(E_\alpha)>0$.
% \end{theorem}
% 
% Theorem \ref{THMUBIQUITY} has been proven powerful for instance:
% \begin{itemize}
% \item in diophantine approximation, to get results improving the famous Jarnik theorem 
% (see \cite{Mat95,BS11});
% \item to determine the size of the level sets of the H\"older exponent of a typical function in a Besov space (see \cite{Jaf00a}).
% \end{itemize}
% In our context, it is exactly what we need to prove Theorem \ref{THMLP}. However, it is not powerful enough for Theorem \ref{THMCT3},
% where we need to give a lower bound for the precised Hausdorff dimension. Thus, we are going to prove an "abstract" version of the
% ubiquity Theorem \ref{THMUBIQUITY}.

\begin{theorem}[The mass transference principle]\label{THMUBI2}
Let $(x_n)_{n\ge 0}$ be a sequence of points in $[0,1]^d$ and let $(r_n)_{n\ge 0}$
be a sequence of positive real numbers decreasing to 0. Let also $\phi:\mathbb R_+\to\mathbb R_+$ 
be a dimension function satisfying $\phi(s)\gg s^d$ when $s$
goes to 0 and $s^{-d}\phi$ is monotonic. Define
\begin{eqnarray*}
E&=&\limsup_n B(x_n,r_n)\\
E^\phi&=&\limsup_n B\left(x_n,\phi^{-1}(r_n^d)\right)
\end{eqnarray*}
and suppose that almost every point of $[0,1]^d$ (in the sense of the
Lebesgue's measure) lies in $E$. Then,
$\mathcal{H}^\phi(E^\phi)=+\infty$.
\end{theorem}
% When $\phi(s)=s^{d/\alpha}$, we get Theorem \ref{THMCT3} with a slight improvement (we can choose
% $\psi(s)=s^{d/\alpha}\big(\ln 1/s\big)^{\veps}$ for every $\veps>0$ for instance). We postpone the proof
% of Theorem \ref{THMUBI2} to Section \ref{SECUBI}.

We shall use it in the following situation.
\begin{corollary}\label{CORUBI2}
 Let $(q_n)$ be a sequence of integers and, for each $n\in\mathbb N$, each $k\leq q_n$, let
 $B_{k,n}=B(x_{k,n},r_n)$ be a ball with center $x_{k,n}\in[0,1]^d$ and with radius $r_{k,n}$
 such that $\lim_{n\to+\infty}\max_k(r_{k,n})=0$. Let also 
 $\phi:\mathbb R_+\to\mathbb R_+$ 
be a dimension function satisfying $\phi(s)\gg s^d$ when $s$
goes to 0 and $s^{-d}\phi$ is monotonic. Define
$$\begin{array}{rclcrcl}
 B_n&=&\bigcup_{k=1}^{q_n}B_{k,n}&\quad&E&=&\limsup_n B_n\\
 B_n^\phi&=&\bigcup_{k=1}^{q_n}B(x_{k,n},\phi^{-1}(r_{k,n}^d))&\quad&E^\phi&=&\limsup_n B_n^\phi.
\end{array}
$$
Suppose that almost every point of $[0,1]^d$ (in the sense of the
Lebesgue's measure) lies in $E$. Then,
$\mathcal{H}^\phi(E^\phi)=+\infty$.
\end{corollary}
\begin{proof}
 Reordering the sequences $(B_{k,n})$ and $(B_{k,n}^\phi)$ as $(C_j)$ and $(C^\phi_j)$, we can observe that
 \begin{eqnarray*}
\limsup_n B_n&=&\limsup_j C_j=E\\
 \limsup_n B_n^\phi&=&\limsup_j C_j^\phi=E^\phi.
  \end{eqnarray*}
  
 Thus the corollary follows from a direct application of Theorem \ref{THMUBI2}.
\end{proof}

\section{Multifractal analysis of the divergence of the Fourier series of functions of $L^p(\mathbb T)$}\label{SECLP}
In this section, we shall prove Theorem \ref{THMLP}. Our method, which is inspired by \cite{Jaf00a}, is divided into two parts.
During the first one, we will construct a single function, which we call the saturating function, satisfying the conclusions of
Theorem \ref{THMLP}. During the second one, we will show how to derive a residual set from this single function.

\subsection{The saturating function}\label{SECSATURATINGLP}
Our intention is to construct a function $g$ such that $|S_ng(x)|$ is big when $x$ is close to a dyadic number. The following
definition gives a precise meaning.
\begin{definition}
A real number $x$ is $\alpha$-approximable by dyadics, $\alpha\geq 1$, if there exist two sequence of integers $(k_n)$, $(j_n)$ 
such that
$$\left|x-\frac{k_n}{2^{j_n}}\right|\leq\frac{1}{2^{\alpha j_n}}$$
and $(j_n)$ goes to infinity. The dyadic exponent of $x$ is the supremum of the set of real numbers $\alpha$ such that $x$ is
$\alpha$-approximable by dyadics.
\end{definition}
We denote by
$$D_\alpha=\left\{x\in[0,1];\ \textrm{$x$ is $\alpha$-approximable by dyadics}\right\}.$$
% It follows from Theorem \ref{THMUBIQUITY} that
% $\dim_\mathcal{H}(D_\alpha)\geq1/\alpha$ (in fact
% $\dim_\mathcal{H}(D_\alpha)=1/\alpha$) and more precisely that $\mathcal{H}^\phi(D_\alpha)>0$
% with $\phi(s)=s^{1/\alpha}(\log 1/s)^2$.
It is easy to check that ${\mathcal H}^{\beta}(D_\alpha)=0$ for $\beta>1/\alpha$ so that $\dim_{\mathcal H}(D_\alpha)\le 1/\alpha$. On the other hand, it is well-known that $\dim_{\mathcal H}(D_\alpha) \geq\frac1\alpha$. 
%For
%example, in \cite{DMPV95} or in \cite{Jaf00b}, we can find a weaker version of
%Theorem \ref{THMUBI2} (known as uibiquity theorem)which allow us to prove that
%${\mathcal H}^\phi(D_\alpha)>0$ with $\phi(s)=s^{1/\alpha}(\log(1/s))^2$ and
%is sufficient to get $\dim_{\mathcal H}(D_\alpha)\ge 1/\alpha$. 
Let us nevertheless show how this
follows from Corollary \ref{CORUBI2}. Indeed, $D_\alpha$ can be described as a limsup set:
$$D_\alpha=\limsup_{j\to+\infty}\bigcup_{k=0}^{2^j-1}I_{k,j}^\alpha$$
where the $I_{k,j}$ are the dyadic intervals
$$I_{k,j}=\left[\frac{k}{2^j}-\frac1{2^{j}},\frac{k}{2^j}+\frac1{2^{j}}\right]$$
and 
$$I_{k,j}^\alpha=\left[\frac{k}{2^{j}}-\frac1{2^{\alpha j}},\frac{k}{2^j}+\frac1{2^{\alpha j}}\right].$$
Since $\bigcup_{k=0}^{2^j-1}I_{k,j}\supset [0,1]$, Corollary \ref{CORUBI2} implies that 
$\H^{1/\alpha}(D_\alpha)=+\infty$. 

\medskip

We are going to define $g\in L^p(\TT)$ such that the divergence index of $g$ at $x$ depends on the dyadic exponent of $x$. The greater the dyadic exponent will be, the greater the divergence index of $g$ at $x$ will be. To do this, we will classify the dyadic intervals following their center. Namely, each $k/2^j$ can be uniquely written $K/2^J$
with $K\notin 2\mathbb Z$ and $1\leq J\leq j$ (such a center comes into play from the $J$-th generation). Let 
$\mathcal I_J=\{K/2^J;\ K\notin 2\mathbb Z,\ 0\leq K\leq 2^{J}-1\}$ and 
$$\mathbf I_{J,j}=\bigcup_{\frac{k}{2^j}\in\mathcal I_J}I_{k,j}\quad\quad \mathbf I'_{J,j}=\bigcup_{\frac{k}{2^j}\in\mathcal I_J}2I_{k,j}.$$
Here and elsewhere, when $I$ is an interval and $c$ is a positive real number, $c I$ means the interval with the same
center as $I$
and with length $c |I|$. Observe that, when  $1\le J<j$, the intervals
$2I_{k,j}$, $\frac{k}{2^j}\in\mathcal I_J$ don't overlap and the set $\mathbf I'_{J,j}$ has measure
$2^{J-1}2^{2-j}$. Observe also that when $J$ is small with respect to
$j$, the real numbers $x$ in $\mathbf I_{J,j}$ 
are well-approximated by dyadics $K/2^J$, since $|x-K/2^J|\leq 1/2^j$.
 
We first define a trigonometric polynomial with $L^p$-norm 1 which is almost constant on each
$\mathcal I_{J,j}$ and which is big on $\mathcal I_{J,j}$ when $J$ is small. 
\begin{lemma}\label{LEMBASICLP}
Let $j\geq 1$. There exists a trigonometric polynomial $g_j\in L^p(\TT)$ with spectrum contained in
$[0,j2^{j+1})$ such that
\begin{itemize}
 \item $\|g_j\|_p\leq 1$;
 \item For any $1\leq J\leq j$ and any $x\in \mathbf I_{J,j}$, we can find two
   integers $n_1$ and $n_2$ satisfying 
   $0\le n_1<n_2< j2^{j+1}$ and such that
 $$|S_{n_2}g_j(x)-S_{n_1}g_j(x)|\geq\frac{1}{4j}2^{-(J-j+1)/p}.$$
\end{itemize}
\end{lemma}
\begin{proof}
We set for any $1\leq J\leq j$:
\begin{itemize}
\item $\chi_{J,j}$ a continuous piecewise linear function equal to 1 on $\mathbf I_{J,j}$, equal to 0 outside $\mathbf I'_{J,j}$, and satisfying
$0\leq \chi_{J,j}\leq 1$ and $\|\chi'_{J,j}\|_\infty\le 2^{j}$;
\item $c_{J,j}=\frac{1}{j}2^{-(J-j+1)/p}$ ($c_{J,j}$ is big when $J$ is small);
\item $g_{J,j}=e_{(2J-1)2^{j}}\sigma_{2^j}\chi_{J,j}$.
\end{itemize}
It is straighforward to observe that the spectrum of $g_{J,j}$ is contained in $[n_{J,j},m_{J,j}]$ with
$$\left\{\begin{array}{rcl}
n_{J,j}&=&(2J-1)2^{j}-(2^j-1)\\
m_{J,j}&=&(2J-1)2^{j}+(2^j-1).
\end{array}\right.$$
Thus, the spectra of the $g_{J,j}$, $1\leq J\leq j$ are
disjoint. Moreover, $\Vert g_{j,j}\Vert_p=1$ and for $1\le J<j$,  $\Vert
g_{J,j}\Vert_p\le\Vert\chi_{J,j}\Vert_p\le 2^{(J-j+1)/p}$.

We finally set
$$g_j=\sum_{J=1}^j c_{J,j}g_{J,j}$$
and we claim that $g_j$ is the trigonometric polynomial we are looking
for. First of all, the spectrum of $g_j$ is included in
$[n_{1,j},m_{j,j}]$ which is contained in $[0,j2^{j+1})$. Moreover, the $L^p$
norm of $g_j$ is
$$\|g_j\|_p\leq\sum_{J=1}^j \frac{1}{j}2^{-(J-j+1)/p}\|g_{J,j}\|_p\leq  1.$$
Pick now any $x\in \mathbf I_{J,j}$, $1\leq J\leq j$ so that
\begin{eqnarray*}
|S_{m_{J,j}}g_j(x)-S_{n_{J,j}-1}g_j(x)|&=&|c_{J,j}g_{J,j}(x)|\\
&=&\frac{1}j 2^{-(J-j+1)/p} |\sigma_{2^j}\chi_{J,j}(x)|.
\end{eqnarray*}
Observing that $\chi_{J,j}(x)=1$ and applying the first point of Lemma
\ref{LEMFEJER} to $1-\chi_{J,j}$, we find
$$\left|\sigma_{2^j}\chi_{J,j}(x)\right|\geq1-\left|\sigma_{2^j}(1-\chi_{J,j}(x))\right|\geq \frac14.$$
Thus,
$$|S_{m_{J,j}}g_j(x)-S_{n_{J,j}-1}g_j(x)|\geq \frac{1}{4j}2^{-(J-j+1)/p}$$
and the conclusion follows with $n_2=m_{J,j}$ and $n_1=n_{J,j}-1$.

\end{proof}

We are now ready to construct the saturating function. It is defined by
$$g=\sum_{j\geq 1}\frac1{j^2}e_{j2^{j+1}}g_j.$$
Observe in particular that the functions 
$e_{j2^{j+1}}g_j$ have disjoint spectra (the spectrum of $e_{j2^{j+1}}g_j$ is 
contained in $[j2^{j+1};j2^{j+2})$ ) and that $g$ belongs to $L^p(\TT)$.

We then show that for any $x\in D_\alpha$, $\alpha>1$,
$$\displaystyle\limsup_{n\to+\infty}\frac{\log |S_ng(x)|}{\log
  n}\geq\frac1p\left(1-\frac1\alpha\right).$$ 
Indeed, let $x\in D_\alpha$
and let $\veps>0$ with $\alpha-\veps>1$. We can find integers $K$ and $J$ with
$J$ as large as we want and $K\notin 2\mathbb Z$ such that
$$\left|x-\frac{K}{2^J}\right|\leq\frac1{2^{(\alpha-\veps/2)J}}.$$
We set $j=[(\alpha-\veps/2)J]$ the integer part of  $(\alpha-\veps/2)J$  and $k$ such that $k/2^j=K/2^J$. Hence, 
$$\left|x-\frac{k}{2^j}\right|\leq\frac{1}{2^{(\alpha-\veps/2)J}}\leq \frac{1}{2^{j}}.$$

Using Lemma \ref{LEMBASICLP}, we can find two integers $n_1$ and $n_2$
satisfying  $j2^{j+1}\le n_1< n_2< j2^{j+2}$ and such that
\begin{eqnarray*}
 |S_{n_2}g(x)-S_{n_1}g(x)|&=&\frac{1}{j^2}|S_{n_2}(e_{j2^{j+1}}g_j)(x)-S_{n_1}(e_{j2^{j+1}}g_j)(x)|\\
 &\geq&\frac{1}{4j^3}2^{-(J-j+1)/p}\\
&\geq&\frac{1}{4j^3}2^{\frac1p\left(j-\frac{j+1}{\alpha-\veps/2}-1\right)}\\
&\geq&C2^{\frac1p\left(1-\frac{1}{\alpha-\veps}\right)j}.
\end{eqnarray*}
It follows that we can find $n\in\{ n_1,n_2\}$ such that  $|S_{n}g(x)|\ge
\frac{C}2 2^{\frac1p\left(1-\frac{1}{\alpha-\veps}\right)j}$. Combining the estimates on $n$ and on $|S_ng(x)|$, and since $J$ (hence $j$, hence $n$) can be taken as large as we want, we get that
$$\limsup_{n\to+\infty}\frac{\log |S_ng(x)|}{\log n}\geq\frac1p\left(1-\frac{1}{\alpha-\veps}\right).$$
Since $\veps>0$ is arbitrary, we obtain in fact that
$$\textrm{for any }x\in D_\alpha,\quad \limsup_{n\to+\infty}\frac{\log |S_ng(x)|}{\log n}\geq\frac1p\left(1-\frac1\alpha\right).$$
At this point, it would be nice to get a lower bound for
$\displaystyle\limsup_{n\to +\infty} \frac{\log
  |S_ng(x)|}{\log n}$ for any $x$ with dyadic exponent equal to
$\alpha$. Unfortunately, this does not seem
easy and we will rather conclude by using an argument lying on Hausdorff measures. Indeed, define
\begin{eqnarray*}
D_\alpha^1&=&\left\{x\in D^\alpha;\ \limsup_{n\to +\infty} \frac{\log |S_ng(x)|}{\log n}=\frac1p\left(1-\frac{1}{\alpha}\right)\right\}\\
D_\alpha^2&=&\left\{x\in D^\alpha;\ \limsup_{n\to +\infty}\frac{\log |S_ng(x)|}{\log n}>\frac1p\left(1-\frac{1}{\alpha}\right)\right\}.
\end{eqnarray*}
We have already observed that $\mathcal{H}^{1/\alpha}(D_\alpha^1\cup
D_\alpha^2)=\mathcal{H}^{1/\alpha}(D_\alpha)=+\infty$. It suffices 
to prove that $\mathcal{H}^{1/\alpha}(D_\alpha^2)=0$. 
Let $(\beta_n)$ be a sequence of real numbers such that $\displaystyle\beta_n
>\frac1p\left(1-\frac1\alpha\right)$ and $\displaystyle\lim_{n\to
  +\infty}\beta_n=\frac1p\left(1-\frac1\alpha\right)$. 

Let us observe that
$$D_\alpha^2\subset \bigcup_{n\ge 0}\mathcal E(\beta_n,g).$$
Moreover, Theorem \ref{THMAUBRY} implies that  $\mathcal{H}^{1/\alpha}(\mathcal E(\beta_n,g))=0$ for
all $n$. Hence, $\mathcal{H}^{1/\alpha}(D_\alpha^2)=0$ and $\mathcal{H}^{1/\alpha}(D_\alpha^1)=+\infty$, which proves
that
$$\dim_\mathcal{H}\left(E\left(\frac1p\left(1-\frac{1}{\alpha}\right),g\right)\right)
\ge\frac1\alpha.$$
By Theorem \ref{THMAUBRY} again, this inequality is necessarily an equality. 
Finally, $g$ satisfies the conclusions of Theorem \ref{THMLP}, setting $1-\beta p=1/\alpha$.

\begin{remark}
If $\alpha=1$, then $\beta=0$ and the conclusion is a
consequence of Carleson's Theorem.
\end{remark}

\subsection{The residual set}
To build the dense $G_\delta$-set, the idea is that any function whose Fourier
coefficients are sufficiently
close to those of the saturating function $g$ on infinitely many intervals $[j2^{j+1};j2^{j+2})$ will 
satisfy the conclusions of Theorem \ref{THMLP}. Precisely,
let $(f_j)_{j\ge 1}$ be a dense sequence of polynomials in $L^p(\TT)$ with
spectrum contained in $[-j,j]$.
 We define a sequence $(h_j)_{j\ge 1}$ as follows:
$$h_j=f_j+\frac{1}{j}e_{j2^{j+1}}g_j$$
so that $\|h_j-f_j\|_p$ goes to 0 and $(h_j)_{j\ge 1}$ remains dense in
$L^p(\TT)$. Observe also that the spectra of $f_j$ and $h_j-f_j$ don't
overlap. Finally, let $(r_j)_{j\ge 1}$ be a
sequence of positive integers so small that, for any $f\in L^p(\TT)$
with $\|f\|_{L^p}\leq r_j$, $\|S_n f\|_\infty\leq 1$ for any $n\le j2^{j+2}$. 
The dense $G_\delta$ set we will consider is
$$A=\bigcap_{l\in\mathbb N}\bigcup_{j\geq l}B(h_j,r_j).$$
Let $f$ belong to $A$ and let $(j_l)_{l\ge 1}$ be an increasing sequence of integers such that $f$ belongs to $B(h_{j_l},r_{j_l})$ for any $l$. Then,
for any $\alpha> 1$, we define $J_l=[j_l/\alpha]+1$ (which is smaller than
$j_l$ if $l$ is large enough) and 
$$E=\limsup_{l \to+\infty} \mathbf I_{J_l,j_l}.$$
For any $x\in E$ one can find $j=j_l$ as large as we want, the
corresponding $J=J_l$ and $1\leq k\leq 2^j-1$ such that $x$ belongs to $I_{k,j}$
with $k/2^j\in \mathcal I_{J}$. 

Observe that $f=f_j+\frac1je_{j2^{j+1}}g_j+(f-h_j)$. 
By Lemma \ref{LEMBASICLP}, we can find two
integer $n_1$ and $n_2$ satisfying $j2^{j+1}\le n_1< n_2< j2^{j+2}$ and such
that
$$|S_{n_2}(e_{j2^{j+1}}g_j)(x)-S_{n_1}(e_{j2^{j+1}}g_j)(x)|\ge
\frac{1}{4j}2^{-(J-j+1)/p}.$$
Using the definition of the $r_j$, we obtain
\begin{eqnarray*}
 |S_{n_2}f(x)-S_{n_1}f(x) |&\geq&\frac1{4j^2}2^{-(J-j+1)/p}-|S_{n_2}(f-h_j)(x)|-|S_{n_1}(f-h_j)(x)|\\
 &\geq&\frac{1}{4j^2}2^{-(J-j+1)/p}-2
\end{eqnarray*}
so that
$$|S_{n_2}f(x)|\ge\frac{C}{j^2}2^{-(J-j+1)/p}\quad\text{or}\quad |S_{n_1}f(x)|\ge\frac{C}{j^2}2^{-(J-j+1)/p}.$$
Observing that
$$
\left\{\begin{array}{rcl}
\max(\log n_2,\log n_1)&=&j\log 2+O(\log j)\\
\null\\
\log\left(j^{-2}2^{-(J-j+1)/p}\right)&=&\frac1p
\left(1-\frac1\alpha\right)j\log2+O(\log j)
\end{array}\right.$$
we find in particular that, for any $x\in E$,
$$\limsup_{n\to +\infty} \frac{\log |S_nf(x)|}{\log n}\geq 
\frac1p\left(1-\frac 1\alpha\right).$$
On the other hand, let us write
$$\mathbf I_{J_l,j_l}=\bigcup_{1\le K<2^{J_l},\ K\notin 2{\mathbb Z}}\left[\frac{K}{2^{J_l}}-\frac1{2^{j_l}},\frac{K}{2^{J_l}}+\frac{1}{2^{j_l}}\right]$$
and remark that for any $l$, since $J_l\geq j_l/\alpha$,
$$\bigcup_{1\le K<2^{J_l},\ K\notin 2{\mathbb Z}}\left[\frac{K}{2^{J_l}}-\frac1{2^{j_l/\alpha}},\frac{K}{2^{J_l}}+\frac{1}{2^{j_l/\alpha}}\right]\supset [0,1].$$

Hence, we can apply Corollary \ref{CORUBI2} to get $\mathcal{H}^{1/\alpha}(E)=+\infty$. We now conclude exactly as in Section \ref{SECSATURATINGLP}
to get $\mathcal{H}^{1/\alpha}(E^1)=+\infty$, with
$$E^1=\left\{x\in E;\ \limsup_{n\to +\infty} \frac{\log |S_nf(x)|}{\log n}=\frac1p\left(1-\frac1\alpha\right)\right\}.$$
Finally, 
$$\dim_\mathcal{H}\left(E\left(\frac1p\left(1-\frac{1}{\alpha}\right),f\right)\right)
\ge\frac1\alpha$$
and $f$ satisfies the conclusions of Theorem \ref{THMLP}, setting $1-\beta p=1/\alpha$.
\begin{remark}During the construction , we didn't use that the spectra
  of the functions $e_{j2^{j+1}}g_j$ are disjoint, 
because we considered each one separately. We could also define $h_j$ by $h_j=f_j+\frac 1je_{j+1}g_j$. 
\end{remark}
\begin{remark}
The above construction can be carried on $L^1(\TT)$. Namely, for quasi-all
$f\in L^1(\TT)$, we obtain for any $\beta\in [0,1]$, 
$$\dim_{\mathcal H}\left( E\left(\beta,f\right)\right)\geq 1-\beta .$$
However, we cannot go further because Carleson's Theorem dramatically breaks
down in $L^1(\TT)$ and we do not have Theorem \ref{THMAUBRY} at our disposal in this context. 
The study of what happens exactly on $L^1(\TT)$ is a very exciting open question.
\end{remark}

\section{Multifractal analysis of the divergence of the Fourier series of functions of $\mathcal C(\mathbb T)$}\label{SECCT}
We turn to the proof of Theorem \ref{THMCT3}. We follow a strategy close to that of Section \ref{SECLP}. First of all, 
we will give un upper bound for 
the precised Hausdorff dimension of the sets $\mathcal F(\beta,f)$ (hence, of the sets $F(\beta,f)$) for
any $f\in\mathcal C(\TT)$ and any $\beta\in(0,1)$. Second, we will build polynomials with small $L^\infty$-norms
and such that their Fourier series have big partial sums on big intervals. 
These polynomials will be the blocks of our final construction. Working on $\mathcal C(\TT)$ adds several difficulties which will be explained
when we will encounter them.

\subsection{The sets $\mathcal F(\beta,f)$ cannot be too big}
We shall prove the following lemma (recall that $\phi_{s,t}(x)=x^s\exp\left((\log 1/x)^{1-t}\right)$).
\begin{lemma}\label{LEMCTMAJO}
Let $\beta\in(0,1)$ and $f\in\mathcal C(\TT)$. Then, for any $\gamma>1-\beta$, 
$$\mathcal H^{\phi_{1,\gamma}}\big(\mathcal F(\beta,f)\big)=0.$$
In particular, the precised Hausdorff dimension of $\mathcal F(\beta,f)$ cannot exceed $(1,1-\beta)$.
\end{lemma}
\begin{proof}
A key point in Aubry's proof of Theorem \ref{THMAUBRY} is the Carleson-Hunt theorem which asserts that,
for any $g\in L^p(\TT)$, $1<p<+\infty$,
$$\|S^* g\|_{p}\leq C_p\|g\|_{p}\quad\textrm{where}\quad S^*g(x)=\sup_{n\ge 0} |S_ng(x)|.$$
On $\mathcal C(\TT)$, a weak inequality (also due to Hunt) remains valid (see \cite[Theorem 12.5]{Ar02}): there
are two absolute constants $A,B>0$ such that, for every $f\in\mathcal C(\TT)$ and every $y>0$,
$$\mathcal \lambda\big(\left\{x\in\TT\ ;\ S^* f(x)>y\right\}\big)\leq Ae^{-By/\|f\|_\infty}.$$
Here, $\lambda$ denotes the Lebesgue measure on $\TT$. 

So, let $\beta\in(0,1)$ and $f\in\mathcal C(\TT)$. We may assume $\|f\|_\infty\leq 1$. For any $M>0$,
we introduce
$$\mathcal F(\beta,f,M)=\left\{x\in\TT;\ \limsup_{n\to +\infty} \,(\log n)^{-\beta}|S_n f(x)|> M\right\}.$$
Since $\mathcal F(\beta,f)=\bigcup_{M>0} \mathcal F(\beta,f,M)$, we just need to prove that
$\mathcal H^{\phi_{1,\gamma}}\big(\mathcal F(\beta,f,M)\big)=0$ for every $M>0$.
 From now on, we fix some $M>0$.
We pick any $x\in\mathcal F(\beta,f,M)$ and $n_x$ large enough such that 
$$|S_{n_x}f(x)|\geq M(\log n_x)^\beta.$$
Such an inequality remains true in an interval around $x$ whose size is not so small. Precisely, 
because $n_x$ can be assumed to be large and since the $L^1$-norm of the
Dirichlet kernel $D_n$  
behaves like $\frac{4}{\pi^2}\log n$, we may assume that
$\|S_{n_x}f\|_\infty\leq (\log n_x)\|f\|_{\infty}\le \log n_x$.
By Bernstein's inequality, $\|(S_{n_x}f)'\|_\infty\leq n_x \log n_x$. Let
$$I_x=\left[x-\frac{M}{2n_x(\log n_x)^{1-\beta}},x+\frac{M}{2n_x(\log
    n_x)^{1-\beta}}\right].$$ 
For any $y\in I_x$, we get 
\begin{eqnarray}\label{EQLEMCTMAJO1}
|S_{n_x}f(y)|\geq \frac M2(\log n_x)^\beta.
\end{eqnarray}
$(I_x)_{x\in\mathcal F(\beta,f,M)}$ is a covering of $\mathcal F(\beta,f,M)$. 
We can extract a Vitali's covering, namely a countable family of disjoint intervals $I_i$, $i\in\mathbb N$,
of length $\frac M{n_i(\log n_i)^{1-\beta}}$ such that $\mathcal F(\beta,f,M)\subset \bigcup_{i}5I_i$.
Let us finally set, for any $q\geq 1$, $\mathcal U_q=\left\{i;\ 2^{q+1}\geq \frac{M(\log n_i)^\beta}{2}>2^q\right\}$.
Without loss of generality, we may assume the $n_i$ so large that $\bigcup_q \mathcal U_q=\mathbb N$. By applying Hunt's theorem,
$$\mathcal \lambda\left(\left\{x;\ S^* f(x)>2^q\right\}\right)\leq Ae^{-B2^q}.$$
Now, by (\ref{EQLEMCTMAJO1}), the set $\{x;\ S^* f(x)>2^q\}$ contains the disjoint intervals $I_i$, for $i\in\mathcal U_q$.
Thus,
$$\sum_{i\in\mathcal U_q} |I_i|\leq Ae^{-B2^q}.$$
Moreover, for any $i\in\mathcal U_q$, it is not hard to check that
$$|I_i|\geq Ce^{-D2^{q/\beta}}$$
for some positive constants $C,D$ which do not depend on $q$. Picking any
$\alpha$ such that $1-\beta<\alpha<\gamma$, 
we get 
\begin{eqnarray*}
\sum_{i\in\mathcal U_q}\mathcal \phi_{1,\alpha}(5|I_i|)
&=&\sum_{i\in\mathcal U_q} 5|I_i| \exp\left((\log (1/5|I_i|))^{1-\alpha}\right)\\
&\leq&5\left(\sum_{i\in\mathcal U_q}|I_i|\right)\exp\left(\left(D2^{q/\beta}-\log 5C\right)^{1-\alpha}\right)\\
&\leq&5Ae^{-B2^q+D'2^{q(1-\alpha)/\beta}}.
\end{eqnarray*}
Since $1-\alpha<\beta$, this shows that there exists $C_0<+\infty$ such that 
$$\sum_{i\in\mathbb N}\phi_{1,\alpha}(5|I_i|)=\sum_{q\in\mathbb
N}\sum_{i\in\mathcal U_q}\phi_{1,\alpha}(5|I_i|)\le C_0.$$
Remember that $\bigcup_i5I_i$ is a covering of $\mathcal F(\beta,f,M)$ and
that the $I_i$ can be chosen as small as we want. We can then conclude that 
$\mathcal H^{\phi_{1,\alpha}}(\mathcal F(\beta,f,M))\le C_0$. In particular, $\mathcal H^{\phi_{1,\gamma}}\big(\mathcal F(\beta,f,M)\big)=0$,
since $\phi_{1,\alpha}\gg\phi_{1,\gamma}$.
\end{proof}

\begin{remark}
The functions $\phi_{1,\gamma}$, for $\gamma>1-\beta$, are not optimal in the statement of the previous lemma.
We can replace them by any function $\phi(x)=x\left(\exp\big((\log 1/x)^\beta\veps(x)\big)\right)$ with 
$\veps(x)$ goes to 0 as $x$ goes to 0.
\end{remark}

\subsection{The basic construction}
When we try to build explicitely a continuous function whose Fourier series diverges at some point, say 0, a natural way is to consider
polynomials $P$ with small $L^\infty$ norm, and satisfying nevertheless that
$|S_nP(0)|$ is big for some large value of $n$. The easiest examples are 
$$P_N(x)=e_N(x) \sum_{j=1}^N \frac{\sin(2\pi jx)}{j},$$
since the sequence $(\|P_N\|_\infty)_{N\ge 1}$ is bounded whereas $|S_N(P)(0)|\sim \frac12\log N$. Moreover, this example is in some sense optimal
since $\|S_N f\|_\infty\leq C(\log N)\|f\|_{\infty}$ for any $f\in \mathcal C(\TT)$.

In our context, we need to find a polynomial $P$ which satisfies a similar property not only at one point, but on a set which is rather big since at the end we want to construct
sets of divergence with Hausdorff dimension 1. This does not seem to be the case for $P_N$, the reason being that $|(S_NP)'(0)|$ behaves like $N$, which is
much bigger than $S_NP(0)$. 

\smallskip

To tackle this problem, we start from a construction of Kahane and Katznelson in \cite{KK66} (see also \cite{Katz}) which
they use to prove
that every subset of $\TT$ of Lebesgue measure 0 is a set of divergence for $\mathcal C(\TT)$. Since we want to control
both the size of the sets $E$ and the index $n$ such that $S_n P(x)$ becomes larger than some given real number for any $x\in E$,
the forthcoming lemma needs very careful estimations.

\begin{lemma}\label{LEMBASICCT}
Let $\beta\in(0,1)$, $\delta\in(0,1)$ and $K\geq 2$. Then there exist an
integer $k\ge K$, an integer $n$ as large as we want and  
a trigonometric polynomial $P$ with spectrum contained in $[0,2n-1]$ such that
\begin{itemize}
\item $|P(x)|\leq 1$ for any $x\in\TT$;
\item $\log |S_nP(x)|\geq (1-\delta)\beta\log\log n$ for any $x\in I_k^\beta$,
\end{itemize}
where $\displaystyle I_k^\beta=\bigcup_{j=0}^{k-1}\left[\frac jk-\frac{1}{2k\exp\big((\log k)^\beta\big)};\frac jk+\frac{1}{2k\exp\big((\log k)^\beta\big)}\right]$.
\end{lemma}
\begin{proof}
Let us first describe the idea of the proof. We shall construct a trigonometric polynomial $Q$ with spectrum in $[1,n-1]$ and with the following properties:
$|\Im m\ Q|$ is small and $|Q|$ is large on a set $E$. We then set $P=e_{n}\times \Im m\ Q$, so that $\|P\|_\infty$ is small. On the other hand,
writing $Q=\sum_{k=1}^{n-1} a_k e_k$, $2i\Im m\ Q=-\sum_{k=1}^{n-1}\overline{a_k}e_{-k}+\sum_{k=1}^{n-1} a_k e_k$, so that
$$|S_n (P)|=\frac 12\left| \sum_{k=1}^{n-1} \overline{a_k} e_{n-k}\right|=\frac12\left|\sum_{k=1}^{n-1} a_k e_k\right|=\frac12|Q|$$
is large on $E$. The construction of $Q$ will be done by taking $\log f$, the
logarithm of an holomorphic function defined on a neighbourhood of the closed
unit disk $\overline\DD$ (which allows to control
the imaginary part of $\log f$ while the modulus of it can be large), and by taking a Fej\'er sum of $\log f$. 

We now proceed with the details. The proof uses holomorphic functions and it is better to see $\TT$ as the boundary of the unit disk $\DD$.
To avoid cumbersome notations, the letter $C$ will denote throughout the proof a positive and absolute constant, whose value may change from line to line.
Let $k\geq K$ whose value will be fixed later. We set:
\begin{eqnarray*}
\veps&=&\frac1{k\exp\big((\log k)^\beta\big)}\\
z_j&=&e^{\frac{2\pi ij}k},\ j=0,\dots,k-1\\
f(z)&=&\frac1k\sum_{j=0}^{k-1}\frac{1+\veps}{1+\veps-\overline{z_j}z}.
\end{eqnarray*}
$f$ is holomorphic in a neighbourhood of $\overline{\DD}$. We claim that $f$ satisfies the following four properties.
\begin{description}
\item[(P1)]$\forall z\in\overline{\DD}$,\quad $\Re ef(z)\geq C\veps$;
\item[(P2)]$\forall z\in I_k^\beta,\quad |f(z)|\geq\Re ef(z)\geq C\exp\big((\log k)^\beta\big)$;
\item[(P3)]$\forall z\in\TT,\quad |f(z)|\leq C\exp\big((\log k)^\beta\big)$;
\item[(P4)]$\forall z\in\TT,\quad \left|\frac{f'(z)}{f(z)}\right|\leq\frac{C}{\veps^3}$.
\end{description}
Indeed, for any $z\in\overline{\DD}$ and any $j\in\{0,\dots,k-1\}$,
\begin{eqnarray}\label{EQBASIC1}
\Re e\left(\frac{1+\veps}{1+\veps-\overline{z_j}z}\right)=\frac{1+\veps}{|1+\veps-\overline{z_j} z|^2}\Re e\big(1+\veps-z_j\overline{z}\big)\geq\frac{1+\veps}{(2+\veps)^2}\times\veps\geq C\veps,
\end{eqnarray}
which proves \textbf{(P1)}. To prove \textbf{(P2)}, we may assume that $z=e^{2\pi i\theta}$ with 
$\theta\in\left[\frac{-\veps}{2};\frac{\veps}{2}\right]$. Then 
$$\Re e\left(\frac{1+\veps}{1+\veps-\overline{z_0}z}\right)=\frac{1+\veps}{|1+\veps-z|^2}\Re e\big(1+\veps-\overline{z}\big)\geq\frac C\veps.$$
If we combine this with (\ref{EQBASIC1}), we get 
$$\Re e f(z)\ge \frac{C}{k \veps}+\frac{k-1}k C\veps\ge \frac{C}{k \veps}=C\exp\big((\log k)^\beta\big).$$
which gives \textbf{(P2)}. 

Conversely, we want to control $\sup_{z\in\TT}|f(z)|$. Pick any $z=e^{2\pi i\theta}\in\TT$.
By symmetry, we may and shall assume that $|\theta|\leq\frac1{2k}$. Then we get
$$\left|\frac{1+\veps}{1+\veps-\overline{z_0}z}\right|\leq\frac{C}{\veps}.$$
Now, for any $j\in\{1,\dots,k/4\}$, we can write
\begin{eqnarray*}
|1+\veps-\overline{z_j}z|&\geq& |\Im m(\overline{z_j}z)|\\
&\geq&\sin\left(\frac{2\pi j}k-2\pi\theta\right)\\
&\geq&\frac2\pi\times2\pi \left(\frac jk-\theta\right)\\
&\geq&\frac{4}k\left(j-\frac12\right).
\end{eqnarray*}
Taking the sum, 
$$\left|\sum_{j=1}^{k/4}\frac{1+\veps}{1+\veps-\overline{z_j}z}\right|
\leq \frac{k(1+\veps)}4\sum_{j=1}^{k/4}\frac{1}{j-1/2}\leq Ck\log k.$$
In the same way, we obtain
$$\left|\sum_{j=3k/4}^{k-1}\frac{1+\veps}{1+\veps-\overline{z_j}z}\right|
\leq Ck\log k$$
whereas $|1+\veps-\overline{z_j}z|\geq C$ for any $j\in[k/4,3k/4]$, so that
$$\left|\sum_{j=k/4}^{3k/4}\frac{1+\veps}{1+\veps-\overline{z_j}z}\right|
\leq Ck.$$
Putting this together, we get 
$$|f(z)|=\left|\frac1k\sum_{j=0}^{k-1}\frac{1+\veps}{1+\veps-\overline{z_j}z}\right|\leq C\left(\frac{1}{k\veps}+\log k+1\right)\leq C\exp\left((\log k)^\beta\right).$$
Finally, it remains to prove \textbf{(P4)}. We observe that
$$f'(z)=\frac1k\sum_{j=0}^{k-1}\frac{(1+\veps)\overline{z_j}}{(1+\veps-\overline{z_j}z)^2}.$$
We do not try to get a very precise estimate for $|f'(z)|$ (this is not useful for us). We just observe
that $|1+\veps-\overline{z_j}z|^2\geq\veps^2$ for any $j\in\{0,\dots,k-1\}$ and any $z\in\TT$, so that
$$|f'(z)|\leq\frac{C}{\veps^2}.$$
If we combine this with \textbf{(P1)}, we get \textbf{(P4)}.
\medskip

  We are almost ready to construct $P$. The next step is to take $h(z)=\log(f(z))$, which defines a holomorphic function
in a neighbourhood of $\overline{\DD}$ by \textbf{(P1)}. Moreover, $|\Im
m(h(z))|\leq\pi/2$ for any $z\in\overline{\DD}$ and $h(0)=0$. Now, we look at
the function $h$ on the boundary of the unit disk $\DD$, that is we introduce
the function $g(x)=h(e^{2i\pi x})$ defined on the circle 
$\TT=\mathbb R/\mathbb Z$. Properties 
\textbf{(P2)}, \textbf{(P3)} and \textbf{(P4)} can be rewritten as 
\begin{eqnarray*}
\forall x\in I_k^\beta,\quad |g(x)|&\geq&(\log k)^\beta-C\\
\forall x\in\TT,\quad |g(x)|&\leq&(\log k)^\beta+C\\
\forall x\in\TT,\quad |g'(x)|&\leq&Ck^3\exp\big(3(\log k)^\beta\big).
\end{eqnarray*}
Let now $n$ be the smallest
integer such that 
${C}{k^3\exp\big(3(\log k)^\beta\big)}\leq n$. We also have
$\|g'\|_\infty\leq n$ and we can apply the second part of Lemma \ref{LEMFEJER}
to the function $\theta(t)=g(t)-g(x)$ when $x\in I_k^\beta$. Recall that 
$\|\theta\|_\infty\leq 2(\log k)^\beta+C$. We get
$$|\sigma_{n}\theta(x)|\leq \frac{(\log k)^\beta}2+C$$
and we can conclude that
$$|\sigma_{n} g(x)|\geq |g(x)|-|\sigma_{n}\theta(x)|\geq \frac{(\log k)^\beta}2-C.$$
We finally set 
$$P=\frac{2}\pi e_{n}\sigma_n(\Im m g)=\frac{2}\pi e_{n}\Im m (\sigma_n g).$$
It is straightforward to check that $\|P\|_\infty\leq 1$ (recall that
$\sigma_n$ is a contraction on $\mathcal C(\TT)$), and that the spectrum of
$\sigma_n g$ is contained
in $[1,n-1]$ ($\hat g(0)=0$ since $h(0)=0$). Now, the simple algebraic trick exposed at the beginning of the proof shows that
$$|S_nP(x)|=\left|\frac{1}\pi \sigma_n g(x)\right|,$$
so that, for any $x\in I_k^\beta$, 
$$|S_nP(x)|\geq \frac1{2\pi}(\log k)^\beta-C.$$
This leads to 
$$\log |S_nP(x)|\geq\beta\log\log k-C.$$
On the other hand,
\begin{eqnarray*}
\log \log n&\leq&\log\left(3\log k+3(\log k)^\beta+\log C\right)\\
&\leq&\log\log k+C.
\end{eqnarray*}
Finally,
$$\frac{\log\log |S_nP(x)|}{\log\log n}\ge\frac{\beta\log\log k -C}{\log\log
  k+C}\ge (1-\delta)\beta,$$
provided $k$ has been chosen large enough. Moreover, $n$ can be chosen as
large as we want since $n\to +\infty$ when $k\to +\infty$.
\end{proof}
\begin{remark}
The fact that we have to compare $\log\log n$ and $\log |S_n|$ helps us for the previous proof. Even if $n$ and $k$ do not have the same order of growth, 
this is not apparent when we apply the iterated logarithm.
\end{remark}
\begin{remark}\label{REMARKNK}
During the construction, the integers $k$ and $n$ can't be chosen
independently~: they satisfy $n-1\le Ck^3\exp\big(3(\log k)^\beta\big)\leq n$
where $C$ is an absolute constant. If we want to construct a polynomial $P$
satisfying the conclusion of Lemma \ref{LEMBASICCT} with a large value of $n$,
we need also to choose a large value of $k$.
\end{remark}
\subsection{The conclusion}
We are now going to prove the full statement of Theorem \ref{THMCT3}. At this point, the situation is less favourable than in the
$L^p$-case. There, the basic construction done at each step $j$ did not depend on the index of divergence that we would like to get.
We had the same function $g_j$ which worked for all indices of divergence, and it was the dyadic exponent of $x$ which
decided how large was $|g_j(x)|$.
The construction done in Lemma \ref{LEMBASICCT} is less efficient, because the polynomial $P$ does depend on the expected divergence index $\beta$
(the index $\beta$ is a parameter of the definition of $f$ above). We have to overcome this new difficulty
and the solution will be to introduce redundancy in the construction of the $G_\delta$-set.

As usual, we start from a sequence $(f_j)_{j\ge 1}$ of polynomials which is dense in $\mathcal C(\TT)$. For convenience, 
we assume that $\|f_j\|_\infty\leq j$ for any $j$ and that the spectrum of $f_j$ is contained in $[-j,j]$.
Furthermore, we fix four sequences $(\alpha_l),\ (\beta_l),\ (\delta_l)$ and $(\veps_l)$ with values in $(0,1)$ and such that:
\begin{itemize}
\item $(\beta_l)$ is dense in $(0,1)$ and $l\mapsto\beta_l$ is one to one;
\item $\sum_l \veps_l\leq 1$;
\item $(\delta_l)$ and $(\alpha_l)$ go to zero.
\item $\delta_l<1/3$.
\end{itemize}
Let now $j\geq 1$. By induction on $l=1,\dots,j$, we build sequences $(P_{j,l})$, $(n_{j,l})$ and $(k_{j,l})$ 
satisfying the conclusions of Lemma \ref{LEMBASICCT} with $\beta=\beta_l$,
$\delta=\delta_l$ and $K=j$ (to ensure that $\lim_{j\to
  +\infty}k_{j,l}=+\infty$) and we will decide how large should be $n_{j,l}$ during
the construction. According to Remark \ref{REMARKNK}, these constraints on
$n_{j,l}$ will
determine the values of the $k_{j,l}$. We then set
$$g_j:=f_j+\alpha_j\sum_{l=1}^j \veps_l e_{n_{j,l}}P_{j,l}$$
so that $\|g_j-f_j\|_\infty\leq \alpha_j\sum_{l=1}^j \veps_l \|P_{j,l}\|_\infty\leq \alpha_j$. In particular, the sequence 
$(g_j)$ remains dense in $\mathcal C(\TT)$. 

Recall that the spectrum of $f_j$ is included in $[-j,j]$ and observe that 
the spectrum of $e_{n_{j,l}}P_{j,l}$ lies in $[n_{j,l},3n_{j,l}-1]$. If we
suppose that $n_{j,1}=j+1$ and $n_{j,l+1}\ge 3n_{j,l}$, we can conclude that
the spectra of $f_j,e_{n_{j,1}}P_{j,1},\cdots,e_{n_{j,j}}P_{j,j}$ are disjoint.

Let now $x$ belongs to $I_{k_{j,l}}^{\beta_{l}}$ for some $l\leq j$. Then
\begin{eqnarray*}
\left|S_{2n_{j,l}}g_j(x)\right|&\geq&\alpha_j \veps_l \left|S_{n_{j,l}}P_{j,l}(x)\right|-\alpha_j\sum_{m=1}^{l-1}\veps_m \|P_{j,m}\|_\infty-j\\
&\geq&\alpha_j\veps_l \left|S_{n_{j,l}}P_{j,l}(x)\right|-\alpha_j-j.\\
\end{eqnarray*}
Because we can choose $n_{j,l}$ as large as we want in the process, 
we may always assume that the choice that we have done ensures that
$$\left|S_{2n_{j,l}}g_j(x)\right|\geq\frac{\alpha_j\veps_l}2 \left|S_{n_{j,l}}P_{j,l}(x)\right|.$$
Taking the logarithm, we find
\begin{eqnarray*}
\log \left|S_{2n_{j,l}}g_j(x)\right|&\geq&\log \left|S_{n_{j,l}}P_{j,l}(x)\right|+\log \veps_l+\log \alpha_j-\log 2\\
&\geq&(1-\delta_l)\beta_l \log\log(n_{j,l})+\log \veps_l+\log\alpha_j-\log 2\\
&\geq&(1-2\delta_l)\beta_l \log\log(2n_{j,l})
\end{eqnarray*}
provided again that we have chosen $n_{j,l}$ very large. 

We then fix $r_j>0$ so small that, for any $f\in B(g_j,r_j)$ 
(the balls are related to the norm $\|\cdot\|_\infty$), for any $l\leq j$, 
$$\|S_{2n_{j,l}}f-S_{2n_{j,l}}g_j\|_\infty\le 1/2.$$ 
Observe that for every
real number $t \ge 1$, we have  
$\log(t-1/2)\ge \log(t)-\log 2$. 
For any $x\in I_{k_{j,l}}^{\beta_{l}}$ with $l\le j$,  we get 
\begin{eqnarray*}
\log \left|S_{2n_{j,l}}f(x)\right|&\ge&\log
\left|S_{2n_{j,l}}g_j(x)\right|-\log 2\cr
&\ge&(1-2\delta_l)\beta_l\log \log (2n_{j,l})-\log2\cr
&\ge&(1-3\delta_l)\beta_l\log \log (2n_{j,l})
\end{eqnarray*}
if $n_{j,l}$ are chosen sufficiently large such that $\delta_l\beta_l\log \log
(2n_{j,l})\ge\log 2$. 

We finally set 
$$A=\bigcap_{p\in\NN}\bigcup_{j\geq p}B(g_j,r_j),$$ 
and we claim that $A$
is the dense $G_\delta$ set we are looking for. 

Indeed, let $f$ belong to $A$ and let $(j_p)$ be an increasing sequence
of integers such that for every $p\ge 0$, $f\in B(g_{j_p},r_{j_p})$. We consider
$\beta\in(0,1)$ and choose $p_0$ such that
$$\left\{\beta_1,\cdots,\beta_{j_{p_0}}\right\}\cap\, (0,\beta)\not=\emptyset .$$

Such a $p_0$ exists since the sequence $(\beta_l)_{l\ge 1}$ is dense in
$(0,1)$. For every $p\ge p_0$, let $l_p$ be chosen in $\{1,\cdots, j_p\}$
such that 
$$\beta-\beta_{l_p}=\inf\{\beta-\beta_l;\ l\leq j_p\textrm{ and }\beta >\beta_l\}.$$
Since the sequence $(\beta_l)$ is dense in $(0,1)$, $\beta_{l_p}<\beta$ for $p\geq p_0$
and $\beta_{l_p}\to \beta$. Moreover, since $l\mapsto \beta_l$ is one to one,
it is clear that $l_p$ is non decreasing and goes to $+\infty$.

Observe that, for $p\geq p_0$, $I_{k_{j_p,l_p}}^{\beta}\subset I_{k_{j_p,l_p}}^{\beta_{l_p}}$, so that,
for any $x\in I_{k_{j_p,l_p}}^{\beta}$, setting $N_p=2n_{j_p,l_p}$,
$$\log |S_{N_p}f(x)|\geq (1-3\delta_{l_p})\beta_{l_p}\log\log(N_p).$$
In particular, setting $F=\limsup_p I_{k_{j_p,l_p}}^\beta$, we get that
$$\limsup_{n\to+\infty}\frac{\log |S_n f(x)|}{\log\log n}\geq \beta$$
for any $x\in F$. Now, we can apply Corollary \ref{CORUBI2} with a jauge
function $\phi$ satisfying 
$\phi^{-1}(y)=y\exp\left[-(\log(1/2y))^{\beta}\right]$ to obtain ${\mathcal
  H}^\phi(F)=\infty$. 

Observe that if $y=\phi(x)$, then
$$y=x\exp\left[(\log(1/2y))^{\beta}\right]
\quad\text{and}\quad \log x\le \log y.$$
It follows that $\phi(x)\le
x\exp\left[(\log(1/2x))^{\beta}\right]\le\phi_{1,1-\beta}(x)$ and  
${\mathcal H}^{\phi_{1,1-\beta}}(F)=+\infty$.

We now conclude exactly as in the $L^p$-case, using Lemma \ref{LEMCTMAJO} to replace Aubry's result. Namely, we set 
\begin{eqnarray*}
F^1&=&\left\{x\in F;\ \limsup_{n\to+\infty}\frac{\log |S_n f(x)|}{\log\log n}= \beta\right\}\\
F^2&=&\left\{x\in F;\ \limsup_{n\to+\infty}\frac{\log |S_n f(x)|}{\log\log n}> \beta\right\}
\end{eqnarray*}
and we observe that Lemma \ref{LEMCTMAJO} guarantees that ${\mathcal H}^{\phi_{1,1-\beta}}(F^2)=0$. Thus,
${\mathcal H}^{\phi_{1,1-\beta}}(F^1)=+\infty$ and the precised Hausdorff dimension of $F(\beta,f)$, which contains $F^1$, is at least $(1,1-\beta)$.
By Lemma \ref{LEMCTMAJO}, it is exactly $(1,1-\beta)$. 

\begin{remark}
 It is amazing that, with our method, it is easier to prove Theorem \ref{THMCT3} and to deduce Theorem
  \ref{THMCT2} from it than to prove Theorem \ref{THMCT2} directly. Indeed, to ensure that the sets
   $\mathcal F(\beta,f)$ are big, we need to know that the sets $F(\beta',f)$ are small for $\beta'>\beta$. This cannot be done if we stay within the notion of Hausdorff dimension.
\end{remark}

\begin{remark}
The method developed above allows us to construct a 
``concrete function'' that satisfies the conclusion of Theorem \ref{THMCT3}. 
More precisely, it suffices to consider
$$g=\sum_{j=1}^{+\infty}\frac1{j^2}\sum_{l=1}^j\veps_le_{n_{j,l}}P_{j,l}$$
with the constraint $3n_{j,j}<n_{j+1,1}$ to ensure that the blocks 
$\sum_{l=1}^j\veps_le_{n_{j,l}}P_{j,l}$ have disjoint spectra. Such a function
is some kind of saturating function in the continuous case.
\end{remark}

\end{document}